\documentclass[11pt]{amsart}
\usepackage{am ssymb, stmaryrd}
\usepackage{hyperref}
\usepackage[margin=1.0 in]{geometry}
\usepackage{color,soul}
\usepackage{multicol}
\usepackage{xcolor}
\usepackage{enumerate}
\usepackage{enumitem}
\usepackage{mdframed}
\usepackage[T1]{fontenc}{}
\normalfont

\theoremstyle{plain}
\newtheorem{theorem}[subsection]{Theorem}
\newtheorem{proposition}[subsection]{Proposition}
\newtheorem{lemma}[subsection]{Lemma}

\newtheorem{corollary}[subsection]{Corollary}

\theoremstyle{definition}

\newtheorem{example}[subsubsection]{Example}

\newtheorem{remark}[subsection]{Remark}

\newcommand{\idealm}{\mathfrak{m}}

\title{Cubic Surfaces of Characteristic Two}
\author{Zhibek Kadyrsizova, Jennifer Kenkel, Janet Page, Jyoti Singh, \protect\\
Karen E. Smith, Adela Vraciu,  and Emily E. Witt}
\email{zhibek.kadyrsizova@nu.edu.kz, jke295@g.uky.edu, jrpage@umich.edu, jyotijagrati@gmail.com, \\ kesmith@umich.edu, vraciu@math.sc.edu, witt@ku.edu}
\thanks{This project was partially supported by the  National Science Foundation (grant number  1934391),  the Banff International Research Station (workshop 19w5104), and the Association for Women in Mathematics (grant number NSF-HRD 1500481). In addition, Jyoti Singh was partially supported by SERB(DST) grant number ECR/2017/000963, Karen E. Smith was partially supported by NSF grant numbers 1801697, 1952399, and 2101075, and Emily Witt was partially supported by  NSF CAREER grant 1945611.}
\begin{document}
\maketitle

\begin{abstract}{Cubic surfaces in characteristic two are investigated from the point of view of prime characteristic commutative algebra.  In particular, we prove that, 
the non-Frobenius split cubic surfaces form a linear subspace of codimension four in the 19-dimensional space of all cubics, and that up to projective equivalence, there are finitely many non-Frobenius split cubic surfaces.  We explicitly describe defining equations for each and characterize them  as extremal in terms of configurations of lines on them. In particular,  a (possibly singular) cubic surface in characteristic two  fails to be Frobenius split if and only if no three lines on it form a "triangle".}
\end{abstract}

\section{Introduction}

 Cubic surfaces have  fascinated mathematicians for nearly  two centuries, going back at least to Cayley and Salmon's  1849  discovery of their famous twenty-seven lines.
  Yet new discoveries about cubic surfaces continue to emerge.  For example, Dolgachev and Duncan recently described the automorphism groups of smooth cubic surfaces in prime characteristic, including a detailed investigation 
 in the oft-overlooked case where the ground field has  {\it characteristic two} (see \cite{dolgachev-duncan.automorphisms}).

 In this paper, we study cubic surfaces, including the singular ones,  through the special lens of characteristic $p$ commutative algebra,
 focusing on the especially interesting case where the ground field  has  {\it  characteristic two}. 
More specifically,  we study  when  cubic surfaces are Frobenius split (in the sense of \cite{mehta-ramanathan-frobenius-splitting}) or equivalently, when their homogeneous coordinate rings are $F$-pure (in the sense of \cite{HR76}).     We show that the vast majority  of cubic surfaces in characteristic two are  Frobenius split. Indeed,   we  explicitly classify the finitely many  non-$F$-pure cubic surfaces (including the singular ones) up to projective change of coordinates, both by giving explicit equations, and in terms of the configuration of lines on them. (Smooth cubic surfaces of characteristic $p>2$ are always Frobenius split  by  \cite[5.5]{hara.rational-singularities}. See Remark \ref{p>2} for a new proof deducing this from the main theorem of   \cite{beauville.sur_les_hypersurfaces}.)

 To state our  results more precisely, 
 fix an algebraically closed field of characteristic two. 
First,  recall that the set of all cubic surfaces in $\mathbb P^3$ is parametrized by the nineteen-dimensional projective space of all cubic forms in four variables.
Most cubic surfaces in this family are Frobenius split: 
the set of cubic surfaces that are not form a fifteen-dimensional linear subspace of this $\mathbb P^{19}$---hence, a proper Zariski closed set of codimension four.  For details, see Proposition \ref{parameter}.

Next,  up to projective  linear change of coordinates,  we show in Theorem \ref{main} that 
    there are precisely five isomorphism types of  non-Frobenius split  cubic surfaces that are not simply cones over planar cubic curves; these are represented by the following equations: 
     \begin{enumerate}
\item $x_1^3+x_2^3+x_3^3+x_4^3$
\item $ x_1^2x_4 + x_2^3 + x_1 x_3^2$
\item $x_1^2x_4 + x_2^2x_3+ x_1 x_3^2$
\item 
    $ x_1^2x_4 + x_2^3 + x_3^3$
\item $x_1^2x_3  + x_2^2 x_4$
\end{enumerate}

Only the first of these is smooth: there is exactly one smooth cubic surface that is not Frobenius split,  a higher-dimensional analog of the fact that (over an algebraically closed field of characteristic two), there is \emph{only one} supersingular elliptic curve  \cite[p. ~260]{Husemoller}.   Indeed, we  deduce the former from the latter, giving a different proof than in  \cite[5.5]{hara.rational-singularities}.

To complete the classification, we  consider cubic surfaces that are  cones over a cubic curve in $\mathbb P^2$. Again, the generic one is Frobenius split---the non-Frobenius split ones are parametrized by a  hyperplane in the $\mathbb P^9$ of all  plane cubics. Again, up to projective linear changes of coordinates, we establish in Proposition \ref{3variable} and Lemma \ref{2variable} that there are only finitely  many non-Frobenius split ones, which can be enumerated as follows:
  \begin{enumerate}
\item  The unique smooth supersingular elliptic curve (projectively equivalent to the Fermat cubic $x^3+y^3+z^3)$.
\item The cuspidal cubic curve (projectively equivalent to $x^2z + y^3$).
\item A line tangent to a smooth conic (projectively equivalent to $x^2z + xy^2$).
 \end{enumerate}
 In addition to these, there exist three distinct  configurations of lines that are not Frobenius split:
 \begin{enumerate}
 \item Three different lines meeting at one point
 (projectively equivalent to $xy(x+y)$).
 \item The union of a line and a double line (projectively equivalent to $x^2y$).
 \item A triple line (projectively equivalent to $x^3$).
 \end{enumerate}
 Note that ``triangles'' of lines---three coplanar lines  that do not meet at a point--- are excluded from the list; such a triangle of lines is always Frobenius split. 
 This observation is crucial to our  characterization of  non-Frobenius split cubic surfaces as those that are \emph{maximally degenerate} from the point of view of the configuration of lines on them.

 More precisely,   we show in Theorem \ref{eckardt}  that a (possibly singular) cubic surface is Frobenius split unless {\it every} pair of  intersecting lines on it  meets in an Eckardt point (or consists of double or triple lines in the singular case).    In the smooth case we can say simply that a smooth cubic surface is Frobenius split unless it contains {\it no triangles.}  This corollary  in the smooth case  also follows by combining  \cite[5.5]{hara.rational-singularities} and  \cite[1.1]{Homma};
 see also \cite[20.2]{Hirschfeld}.

  Our method for studying non-$F$-pure cubics in characteristic two can be described as {\it linear algebraic}, different from the algebro-geometric or commutative algebraic approaches of \cite{Hirschfeld}, \cite{Homma} and \cite{hara.rational-singularities}. We show that each such cubic form can be represented by a unique matrix, and we can explicitly describe the action of the group of coordinate changes on such forms in terms of this matrix; see Section \ref{linalg}. Related techniques are considered by Lang  in his study of algebraic groups over finite fields; see \cite{lang.algebraic_groups}. 
  
  {} 
  We work over a fixed algebraically closed field $k$ of characteristic two, except where otherwise indicated.
  
  {}
  \noindent
  {\bf Acknowledgements.} This paper began at a weeklong research workshop called {\it Women in Commutative Algebra}  at the Banff International Research Station in October 2019 and partially funded by US NSF   and the AWM.  The fifth author is grateful to the organizers of the conference  {\it Del Pezzo surfaces and Fano varieties} 
the previous July at Heinrich-Heine-Universit\"at in D\"usseldorf, whose stimulating environment inspired this project idea. In particular, the fifth author acknowledges several fruitful and lively discussions with Igor Dolgachev and with Damiano Testa there. She also thanks Izzet Coskun for suggesting the first proof of Corollary \ref{EckSmooth}, and S\'andor Kov\'acs for a helpful conversation about Remark \ref{p>2} and for supplying the reference \cite{MumfordSuominen}. All authors also acknowledge discussions with Elo\'isa Grifo, who took part in early discussions about the project.

 \section{Definitions and Preliminary Material}\label{prelim}
 
 A map of commutative rings $A\rightarrow B$ is {\bf pure} if the induced map $M\rightarrow M\otimes_A B$ is injective for all $A$-modules $M$. For example, a map of rings 
 $A\rightarrow B$ is pure if it splits as a map of $A$-modules.  In fact, purity is equivalent to splitting  if $A\rightarrow B$ is finite and $A$ is Noetherian \cite[5.3, 5.5]{HR76}, although  in general, purity of a map is a  weaker  condition.

 For a  commutative ring $R$ of prime characteristic $p$, the Frobenius map   is the ring homomorphism  $R\rightarrow R$ sending each element to its $p$-th power. 
 We say that $R$ is {\bf  $F$-pure}  if its Frobenius map is pure. While formally defined by Hochster and Roberts in \cite{HR76}, 
 $F$-purity  played a starring role in their famous proof of the Cohen-Macaulayness for rings of invariants \cite{hochster+roberts-rings-of-invariants}.

For a scheme $X$ over a field of prime characteristic $p$, we also have a Frobenius map---the scheme  map $X\overset{F}\longrightarrow X$ that is the identity map on the underlying topological space, but the $p$-th power map on functions on each open set. The scheme $X$ is {\bf Frobenius split} if 
 the corresponding map of sheaves $\mathcal O_X\rightarrow F_*\mathcal O_X$ splits. The term ``Frobenius splitting'' was coined by Mehta and Ramanathan, who used it masterfully to prove vanishing theorems for cohomology of sheaves on Schubert varieties. 
 
 Frobenius splitting for an affine variety $X$ is the same as the $F$-purity of its coordinate ring by the aforementioned result in  \cite[5.3, 5.5]{HR76}, because the Frobenius map is always finite for a finitely generated algebra over  an algebraically closed field.
 Likewise,
  Frobenius splitting for  a projective variety $X$  is equivalent to the $F$-purity of  any (equivalently, every) homogeneous coordinate ring for a projectively normal embedding of $X$ into projective space, or more generally, for any (equivalently, every) {\it section ring} of $X$; see \cite[4.2]{smith.vanishing}, \cite[3.10]{smithGlobalF-regularityGITQuotientsFanoVarieties}, or \cite[1.1.14]{BK}.

 In this paper, we are interested in cubic 
 surfaces---subschemes of $\mathbb P^3$ cut out by a single homogeneous polynomial $f$ of degree 3. 
  A cubic surface $X \subset \mathbb P^3$ over a field of characteristic $p$  is Frobenius split if and only if its homogeneous coordinate ring  $k[x_1, x_2, x_3, x_4]/\langle f\rangle$ is $F$-pure.  Although the terms ``$F$-pure'' and ``Frobenius split''  are essentially equivalent in our context and often used interchangeably, we will  use {\it Frobenius split}  when talking about varieties and {\it $F$-pure}  when talking about rings or forms, in keeping with the historical use of these words.
 
   {}
  There is a convenient criterion for $F$-purity  due to Fedder, which we state  only in the special case we need:

\begin{theorem}  \label{fedder}  \cite[1.2]{fedder.F-purity} 
Given a homogeneous polynomial $f$ in $R = k[x_1, \dots, x_n]$, where $k$ is a perfect field of characteristic $p>0$, the graded  ring  $R/\langle f\rangle$ is $F$-pure if and only if 
$$f^{p-1} \notin \idealm^{[p]}$$ where $\idealm =\langle  x_1, \dots, x_n\rangle$  denotes the unique homogeneous maximal ideal of $R$, and $\idealm^{[p]}
=  \langle x_1^p, \dots, x^p_n\rangle$
 is the ideal generated by the $p$-th powers of the elements of $\idealm$. 
\end{theorem}

{}
\begin{remark} It is easy to see that the ring $  k[x_1, \dots, x_n]/\langle f \rangle$ is $F$-pure if and only if  $  \overline{k} [x_1, \dots, x_n]/\langle f \rangle$ is $F$-pure, where $\overline{k}$ is the algebraic closure of $k$. This follows from Fedder's criterion immediately since $f^{p-1} \notin \idealm^{[p]}$ in the polynomial ring over $k$ 
 if and only if the same is true over $L$, where $L$ is any field extension of $k$.  Alternatively, this is a special case of the more general fact  \cite[5.13]{HR76}.
\end{remark}

{}
\begin{remark} \label{degen} Fix a form $f \in k[x_1, \dots, x_n]$, which we may also consider as a form in one more variable, $\{x_1, \dots, x_n, x_{n+1}\}$. It is  clear from Fedder's criterion that $  k[x_1, \dots, x_n]/\langle f \rangle$ if $F$-pure if and only if $  k[x_1, \dots, x_n, x_{n+1}]/\langle f \rangle$ is $F$-pure. We will use this in the following more geometric form: if a hypersurface $X\subset \mathbb P^n$  is a cone over some variety $Y$ in a lower dimensional projective space, then $X$ is Frobenius split if and only if $Y$ is  Frobenius split.  

\end{remark}

{}
\begin{example}\label{fedder-elliptic: R}  An elliptic curve is  Frobenius split  if and only if  it  is {\it ordinary,} that is,  {\it not  supersingular.} One way to see this is using  Fedder's criterion: the graded ring $k[x, y, z]/\langle f \rangle$ is $F$-pure if and only if $f^{p-1} \notin \langle x^p, y^p, z^p\rangle$, but when $f$ has degree three, this can  happen if and only if the monomial  $(xyz)^{p-1}$ appears in   $f^{p-1}$ with non-zero coefficient. This recovers the well-known criterion for ordinariness of the elliptic curve defined by $f$  \cite[IV  4.21]{hartshorne-algebraic-geometry}.  See \cite[4.3]{smith.vanishing} for a different proof.
\end{example}

We will also need another result of Fedder 
guaranteeing that $F$-purity ``deforms'' in a Gorenstein ring:
 
 \begin{theorem}  \label{deform}   \cite[3.4]{fedder.F-purity}  Let $(R, \mathfrak m)$ be a Gorenstein local (or standard graded) ring and let $f\in\mathfrak m$ be a regular element. If $R/\langle f\rangle$ is $F$-pure, then $R$ is also $F$-pure.
 \end{theorem}

{}
Although the results in this paper do not rely on it, we record here the following theorem of Beauville, which we use to  provide alternative proofs of several steps  throughout the paper:

 \begin{theorem}  \label{beauville}   \cite{beauville.sur_les_hypersurfaces} Let $X\subset \mathbb P^n$ be a smooth hypersurface over an algebraically closed field $k$ of prime characteristic $p > 0$.  If all smooth hyperplane sections of $X$ are isomorphic to one another, then (in suitable linear coordinates) $X$ is defined by an equation of the 
 form $\sum_{i=0}^n x_i^{p^e+1} $ for some non-negative integer $e$, where $p$ is the characteristic of $k$.
 \end{theorem}

\section{$F$-purity of Cubics and some Linear Algebra}\label{linalg}

Fix a field $k$. 
A cubic form in four variables  over $k$ is an element of  the twenty dimensional space ${\text{Sym}}^3\left((k^4)^*\right)$. The corresponding set of cubic surfaces in $\mathbb P^3$ (including all the degenerate, singular, and even non-reduced ones)  is therefore  parametrized by the nineteen dimensional projective space $\mathbb P({\text{Sym}}^3\left((k^4)^*)\right)$.

{}
Explicitly, we can write a cubic form 
(uniquely) 
in the form 
$$
x_1^2 L_1 + x_2^2L_2 + x_3^2L_3 + x_4^2L_4 + b_1x_2x_3x_4 + b_2x_1x_3x_4 + b_3 x_1x_2x_4 + b_4x_1x_2x_3
 $$
where the 
$L_i$ are linear forms and the $b_i$ are scalars. Using matrix notation, we can write
$$
\begin{bmatrix}
L_1 \\
L_2\\
L_3\\
L_4
\end{bmatrix} 
= 
\begin{bmatrix}
a_{11} & a_{12} & a_{13} & a_{14} \\
a_{21} & a_{22} & a_{23}  & a_{24} \\
a_{31} & a_{32} & a_{33} & a_{34} \\
a_{41} & a_{42} & a_{43} &  a_{44} 
\end{bmatrix} 
\begin{bmatrix}
x_1 \\ x_2 \\ x_3 \\ x_4 \\
\end{bmatrix}, $$
and thus see that the sixteen scalars $a_{ij}$, together with the four scalars $b_k$,  determine a unique cubic form in $x_1, \dots, x_4$.  Compactly, we write  the  cubic as \begin{equation}\label{Cubic4}
\begin{bmatrix}   x_1^2 &  x_2^2  &  x_3^2 & x_4^2\end{bmatrix}
 A \begin{bmatrix}  x_1 \\ x_2 \\ x_3 \\x_4\end{bmatrix} + b_1x_2x_3x_4 +  b_2x_1x_3x_4 + b_3 x_1x_2x_4 + b_4x_1x_2x_3
 \end{equation}
 where $A$ is the  $4\times 4$ matrix $[a_{ij}]$.

 Similarly (or as a special case in which scalars associated with $x_4$ are zero), a cubic form  in three variables can be written as 
\begin{equation}\label{Cubic3}
\begin{bmatrix}   x_1^2 &  x_2^2  &  x_3^2  \end{bmatrix}
 A \begin{bmatrix}  x_1 \\ x_2 \\ x_3 \end{bmatrix} + bx_1x_2x_3.
 \end{equation}
where $A$ is a $3\times 3$ matrix and $b$ is a scalar.

{}
Thinking of these  $a_{ij}$ and $b_1, b_2, b_3, b_4$ as the homogeneous coordinates for the $\mathbb P^{19}$ parametrizing all cubic surfaces, we have the following description of the Frobenius split ones in characteristic two: 

\begin{proposition}\label{parameter} Fix a ground field $k$ of characteristic two. 
The set of \textbf{Frobenius split cubic surfaces}  is the non-empty Zariski open set  of
  the $\mathbb P^{19}$ of all cubic surfaces that is 
  complementary to the   codimension four linear subspace 
where $b_1=b_2=b_3=b_4 = 0$ in the expression of the cubic 
as
\begin{equation}
x_1^2L_1+ x_2^2L_2 + x_3^2L_3 + x_4^2L_4 + b_1x_2x_3x_4 +  b_2x_1x_3x_4 + b_3 x_1x_2x_4 + b_4x_1x_2x_3.
 \end{equation}
 Put differently, a cubic surface is Frobenius split if and only if  some $b_i$ is non-zero---that is, if and only if its equation has a non-zero square-free monomial. 
\end{proposition}

\begin{proof}
 In characteristic two, Fedder's criterion, Theorem \ref{fedder}, tells us that $k[x_1, \dots, x_n]/\langle f\rangle$ is $F$-pure if and only if 
$f \notin \idealm^{[2]}$, where $\idealm=\langle  x_1, \dots, x_n\rangle.$ 
Examining Equation \eqref{Cubic4}, we see that this is the same as saying that some $b_i$ is non-zero. 
\end{proof}

{}
\begin{remark}\label{Hart}
As a special case, a cubic in three variables is Frobenius split if and only if, writing it in the form
$$x^2L_1+ y^2L_2 +  z^2L_3 + bxyz,
$$
the  ``square-free''  term $xyz$ appears with non-zero coefficient.  This recovers the fact that a  \emph{smooth} elliptic curve of characteristic $2$   is ordinary if and only if 
the cubic polynomial $f$ defining it as a subvariety of $\mathbb P^2$  has a nonvanishing square free monomial term; see Example \ref{fedder-elliptic: R}. 
\end{remark}

   \subsection{Singular Locus of Non-$F$-pure Cubics}

   For future reference, we record here a simple description of the singular locus for a cubic surface that is not Frobenius split  in terms of the matrix $A$ representing it:
   
 \begin{proposition} \label{singloc}
Given a  non-$F$-pure cubic $h$ over a field of characteristic 2  in four variables, 
$$h=  \begin{bmatrix}   x_1^2 &  x_2^2  &  x_3^2 & x_4^2\end{bmatrix}
 A \begin{bmatrix}  x_1 \\ x_2 \\ x_3 \\ x_4\end{bmatrix},$$
the codimension of the singular locus of the cubic surface defined by $h$ is equal to the rank of $A$.
 \end{proposition}

 \begin{proof} 
 The singular locus is defined by the vanishing of 
 $$
 \frac{\partial h}{\partial x_1}, \,\,\, \frac{\partial h}{\partial x_2}, \,\,\, \frac{\partial h}{\partial x_3}, \,\,\, \frac{\partial h}{\partial x_4},
 $$
 i.e., the system of equations
 $$
 A^{tr} \begin{bmatrix}  x_1^2 \\ x_2^2 \\ x_3^2 \\ x_4^2\end{bmatrix} = 0.
 $$
 This is the ``double linear space'' defined by 
  $$
{A^{[1/2]}}^{tr} \begin{bmatrix}  x_1 \\ x_2\\ x_3\\ x_4\end{bmatrix} = 0,
 $$
 so that its codimension is precisely the rank of $ {A^{[1/2]}}^{tr}$, which is the same as the rank of $A$. Here the notation $A^{[1/2]}$ denotes the matrix obtained from $A$  by taking the unique square root of each entry, and $ {A^{[1/2]}}^{tr}$ is its transpose.
 \end{proof}

{}
\subsection{Changing Coordinates}  
\label{CoordChange} We record some observations about the behavior  under coordinate changes. 

{}
Let  $h$ be a non-$F$-pure cubic defining a cubic surface of  characteristic $2$. Using  Proposition \ref{parameter}, we can write $h$   uniquely as 
\begin{equation}
h =  \begin{bmatrix}   x_1^2 &  x_2^2  &  x_3^2 & x_4^2\end{bmatrix}
 A \begin{bmatrix}  x_1 \\ x_2 \\ x_3 \\ x_4\end{bmatrix}
\end{equation}
where $A$ is some $4\times 4$ matrix of scalars. Let $g = [\lambda_{ij}] \in GL_4(k)$ be an invertible $4\times 4$ matrix  which acts on the coordinates $\{x_1, \dots, x_4\}$ in the obvious way, i.e., via
\[ 
g \cdot  \begin{bmatrix}  x_1 \\ x_2 \\ x_3 \\ x_4 
\end{bmatrix} = \begin{bmatrix}  \sum_{j=1}^4 \, \lambda_{1j} x_j \\ 
 \sum_{j=1}^4  \,  \lambda_{2j} x_j \\
 \sum_{j=1}^4  \,  \lambda_{3j} x_{j} \\   \sum_{j=1}^4 \,  \lambda_{4j} x_{j}
\end{bmatrix}.
\]

{}

Given a matrix $B$ of any size over a field of characteristic $p>0$, 
   we denote by $B^{[p]}$ the matrix obtained by raising each entry of $B$ to the $p$-th power. If $g$ is a change of coordinates represented by an invertible $4\times 4$ matrix,  then it is easy to check that 
 $$g \cdot  \begin{bmatrix}   x_1^p \\  x_2^p  \\  x_3^p \\ x_4^p\end{bmatrix} \, = \, \large[g 
   \begin{bmatrix}  x_1 \\ x_2 \\ x_3 \\ x_4\end{bmatrix}\large ]^{[p]} \, = \,  g^{[p]} \begin{bmatrix}  x_1^p \\  x_2^p  \\  x_3^p \\ x_4^p\end{bmatrix}.$$
  Here the notation $\cdot$ indicates the ring automorphism induced by the linear change of coordinates, and all other adjacent symbols are usual matrix product.

  So our change of  coordinates formula  for a non-$F$-pure cubic is 
  $$
g \cdot  h \, =  \, g \cdot \begin{bmatrix}  x_1^2 &  x_2^2  &  x_3^2 & x_4^2\end{bmatrix}
  A   \begin{bmatrix}  x_1 \\ x_2 \\ x_3\\ x_4 \end{bmatrix} = 
 \begin{bmatrix}  x_1^2 &  x_2^2  &  x_3^2 & x_4^2\end{bmatrix}
\left[g^{[2]}\right]^{tr}  A g  \begin{bmatrix}  x_1 \\ x_2 \\ x_3\\ x_4 \end{bmatrix},
 $$
 where $B^{tr}$ denotes the transpose of the matrix $B$.
 We can write this in the compact form 
\begin{equation}\label{compact1}
g \cdot  h  \, = \, g \cdot \left[(\vec{x}^{[2]})^{tr} A \,  \vec{x}\right]  = 
 (\vec{x}^{[2]})^{tr} 
 \left[
  \left[g^{[2]}\right]^{tr}  A g 
  \right]
  \, \vec{x},
\end{equation}
 or simply state that when a coordinate change $g$ acts on a non-$F$-pure cubic $h$ whose matrix is $A$, the matrix of $g\cdot h$ is
\begin{equation}\label{compact}
  \left[g^{[2]}\right]^{tr}  A g.
\end{equation}

   {}
 
 \begin{remark}\label{rowops}
It is worth recording how each elementary coordinate operation  affects  the matrix $A$ representing a non-$F$-pure cubic $h$. By elementary operation, we mean one of the following:
\begin{itemize}
\item Swap two variables:  $x_i\mapsto x_j$ and $x_j\mapsto x_i$,  fixing the others.
\item Multiply coordinate $x_i$  by a non-zero scalar  $\lambda$:  $x_i\mapsto \lambda x_i$,  fixing the others.
\item Replace $x_i$ by $x_i + \lambda x_j$ for some $j\neq i$, fixing the others. 
\end{itemize}
Each of these corresponds to multiplying the column vector  $\begin{bmatrix}  x_1 \\ x_2 \\ x_3\\ x_4 \end{bmatrix}$ on the left by the corresponding elementary matrix $E$. 
The effect on the $A$ matrix representing  $f$ is to multiply by the transpose of $E^{[2]}$ on the left, and by $E$ on the right.  This amounts to the following, in each respective case: 
\begin{itemize}
\item Swap columns $C_i$ and $C_j$  {\bf and} rows $R_i$ and $R_j$,   fixing the others.
\item Multiply row $R_i$ by $\lambda^2 $  {\bf and}  column $C_i$ by $\lambda$. 
\item Replace column $C_j$ by column $C_j+\lambda C_i$   {\bf and} replace row $R_j$ by row $R_j+\lambda^2 R_i$.

\end{itemize}
 \end{remark}

 {}
 
  \begin{remark}
 We can think of Equation \eqref{compact} as describing a right action $GL_4(k)$  on $k^{4\times 4}$ such that  
 $$
 g \in GL_4(k) \,\,{\text{acts on }} A\in k^{4\times 4} {\text{ by }} \left[g^{[p]}\right]^{tr}A g.
 $$
     Lang studies these representations in \cite{lang.algebraic_groups}.
          \end{remark}
          
     {}
   
     \section{Classification of Non Frobenius split cubic surfaces.}

In this section, we prove that up to projective change of coordinates, there are only finitely many cubic surfaces of characteristic two  that are not Frobenius split, and list them out explicitly. 
We state the classification separately for degenerate and non-degenerate cubic surfaces, before proving all the statements.

By a {\it non-degenerate} form in $n$ variables, we mean a form that can not be written in fewer than $n$ variables after linear change of coordinates. Geometrically, this means the corresponding hypersurface is not the cone over a hypersurface in  a smaller dimensional projective space. 

\begin{theorem}\label{main} Let $h$ be a non-degenerate cubic form  in four variables over an algebraically closed field  of characteristic two. Then the hypersurface in $\mathbb P^3$ defined by $h$ fails to be Frobenius split if and only if, up to linear change of coordinates, $h$ is exactly {\it one} of the following:
  \begin{enumerate}[label=\textup{(\arabic*)}]
\item $x_1^3+x_2^3+x_3^3+x_4^3$
\item $ x_1^2x_4 + x_2^3 + x_1 x_3^2$
\item $x_1^2x_4 + x_2^2x_3+ x_1 x_3^2$
\item 
    $ x_1^2x_4 + x_2^3 + x_3^3$
\item $x_1^2x_3  + x_2^2 x_4$
\end{enumerate}
The first of these is the unique smooth non-Frobenius split cubic surface. The second, third, and fourth  are normal, with an isolated singularity at $[0:0:0:1]$. The final one is the non-normal cubic surface whose singular locus is the line $x_1=x_2=0$. 
\end{theorem}

All these surfaces are extremal from the point of view of the collection of lines on them, as we will prove in  Theorem \ref{eckardt}.

{}
Before embarking on the proof of Theorem \ref{main}, we state and prove the classification for cubic surfaces that are cones over curves in $\mathbb P^2$. In light of Remark \ref{degen}, the classification follows from:

\begin{theorem}\label{3variable} 
A non-degenerate non-$F$-pure cubic in three variables over an algebraically closed field of characteristic two is projectively equivalent to a plane curve defined by one of the following:
\begin{enumerate}[label=\textup{(\arabic*)}]
\item  $x^3+ y^3+z^3$
\item $x^2z + y^3$
\item $x^2z + xy^2$
\end{enumerate}
The first is the only non-$F$-pure smooth cubic curve: every  supersingular elliptic curve is projectively equivalent to this Fermat curve.
The second is  irreducible cuspidal curve cubic with one  singular point (at $[0:0:1]$ in the given coordinates).  The third is  a union of a smooth conic and a line tangent to it at one point ($[0:0:1]$ in the given coordinates). 
\end{theorem}

{}
Finally, we record for future reference the classification of cubics in two variables; this completes the classification of non-Frobenius split cubic surfaces in characteristic two, as this covers the case  where the surface is the cone over a collection of points in $\mathbb P^1$. 

\begin{lemma}\label{2variable}
Every cubic  form $h$  in two variables over an algebraically closed field of any characteristic can be brought to exactly one of the following three  forms by  linear change of coordinates:
    \begin{enumerate}[label=\textup{(\arabic*)}]
    \item $x^3$
    \item $x^2y$
    \item $xy(x+y)$, or any other cubic form with three distinct roots.
    \end{enumerate}
All  of these define non-$F$-pure quotients $k[x, y]/\langle h \rangle$. 
    \end{lemma}
 
\begin{proof}  
  A homogeneous form $h$ in two variables factors completely into three linear forms $L_1L_2L_3$.
The three cases amount to whether the cubic has one, two or three distinct roots.  These are represented by the three forms above because any set of three (or fewer)  points  in $\mathbb P^1 $ are projectively equivalent. Finally,  because   $h^{p-1}$ has degree $3p-3$, it
    is in the ideal $\langle x^p, y^p\rangle $, which means $h$ defines a non-$F$-pure hypersurface by Fedder's criterion. 
    \end{proof}

\subsection{The Proofs of Theorems \ref{3variable} and  \ref{main}}

\begin{proof}[Proof of  Theorem \ref{3variable}]  Let $f$ be a cubic in three variables that is not $F$-pure. 
Such an $f$ can be written uniquely as 
 $$
 f = x^2L_1 + y^2L_2+ z^2L_3 = \begin{bmatrix}   x^2&  y^2  &  z^2 \end{bmatrix}
A  \begin{bmatrix}  x \\ y \\ z \end{bmatrix},
$$
 where the $L_i$ are linear forms in three variables, and the
  $A$ is the matrix whose $i$-th row is the coefficients of the linear form $L_i$. 
  
  If $A$ has rank $3$, then the cubic represents a smooth cubic curve, whence it is a supersingular elliptic curve of characteristic two. There is only one such curve up to isomorphism
   \cite[p 260]{Husemoller}, so  after changing coordinates, we may assume that $f$ is the Fermat cubic, which is non-$F$-pure by Proposition  \ref{parameter}.
   
   If $A$ has rank one, then  after  changing the names of the variables if necessary, we 
  can assume that the second and third rows of $A$ are multiples of the first. So our cubic form is
  $$
  x^2L + y^2(\lambda L) + z^2 (\mu L) = \left(x+ \lambda^{1/2} y + \mu^{1/2} z\right)^2  L
  $$
  for some linear form $L$ and scalars $\lambda, \mu$. Changing coordinates, this is $x^2y$ or $x^3$,  depending on whether or not $L$ is a scalar multiple of $x + \lambda^{1/2} y + \mu_3^{1/2} z$. Both of these cases are degenerate (and hence covered by Lemma \ref{2variable}). 
  
{}
Finally, we consider the case where $A$ has rank two. In this case, we can assume without loss of generality that $f$ can be written 
$$
f =  x^2 L_1 + y^2L_2+z^2(aL_1 + bL_2) =  
(x + a^{1/2}z)^2 L_1 + (y + b^{1/2} z)^2 L_2,
$$
where is $L_1$ and $L_2$ are linear forms and $a$ and $b$ are scalars. Changing coordinates so the rank two form is
$x^2L_1 + y^2L_2$, 
 the bottom row of $A$ can be assumed to be zero.

Next, we analyze the effect of  coordinate changes of $f$ in terms of the matrix $A$ using the technique explained in Subsection \ref{CoordChange}.
Note that in the expression $x^2L_1 + y^2L_2$, at least one of the terms $x^2z$ or $y^2z$ must appear with non-zero coefficient, for otherwise the form is in the two variables $x, y$ and is degenerate. 
Changing the names of the variables (swapping $x$ and $y$) if needed, we can assume that the coefficient of
$x^2z$ is not-zero (and hence scaling, we can assume it is  1, if we'd like). 

Making use of Remark \ref{rowops}, we can add multiples of column 3 to column 1 and to column 2 to clear out the coefficients of $x^3$ and $x^2y$ in the matrix $A$; this and the corresponding row operations do not affect the fact that row 3 consists of zeros. So without loss of generality, we can assume that the matrix has the form 
 $$
 A = 
\begin{bmatrix}
0 & 0 & 1 \\
*  & * & * \\
0 & 0 & 0\\
\end{bmatrix}.
$$
Now adding  a multiple of row 1 to row 2 (and the corresponding column operation), we can assume it is 
 $$
 A = 
\begin{bmatrix}
0 & 0 & 1 \\
a  & b & 0\\
0 & 0 & 0\\
\end{bmatrix}.
$$
If  $a=b=0$, then the form is $x^2z$, which is degenerate, so assume at least one is not zero. If both $a$ and $b$ are non-zero, then add a multiple of column 2 to column 1 to make $a=0$; the corresponding row operation changes row 1, but this can be easily corrected by again clearing out row 1 using column 3 (the corresponding row operation does nothing since row 3 is a zero row).  Thus we can assume that exactly one of $a$ or $b$ is non-zero, in which case  we can scale the appropriate variable to  assume it is 1. 
So, up to linear changes of coordinates, every rank two (non-degenerate) cubic in three variables will be represented by one of the following matrices:
$$
\begin{bmatrix}
0 & 0 & 1 \\
0  & 1  & 0\\
0 & 0 & 0\\
\end{bmatrix} \,\,\,\,\,\,\,\,\,\,\,\,\,\,\,\,\,\,
\begin{bmatrix}
0 & 0 & 1 \\
1  & 0 & 0\\
0 & 0 & 0\\
\end{bmatrix}.
$$
These
correspond to the two forms
$$
x^2z+y^3 \,\,\,\,\,\,\,\, {\text{and}} \,\,\,\,\,\,\,\,  x^2z + xy^2
$$
which are distinct, since the former is irreducible while the latter is not.

\end{proof}

   \begin{proof}[Proof of Theorem \ref{main}]
  Let 
   $$f= x_1^2L_1+x_2^2L_2+x_3^2L_2+x_4^2L_4 =  \begin{bmatrix}   x_1^2 &  x_2^2  &  x_3^2 & x_4^2\end{bmatrix}
 A \begin{bmatrix}  x_1 \\ x_2 \\ x_3 \\ x_4\end{bmatrix}$$
 be a non-degenerate  non-$F$-pure cubic in four variables. 
We need to show
 \begin{enumerate}
 \item $f$ can be brought to one of the five normal forms of Theorem \ref{main} by linear change of coordinates
 \item The five polynomials in Theorem \ref{main} define non-isomorphic cubic surfaces.
 \end{enumerate}

 {}
   
 Towards (1), we give separate arguments, depending on the rank of $A$.

We first dispatch with the rank four case---that is, the case where the projective  hypersurface  $\mathcal X$ defined by $f$  is smooth---using our geometric characterization of Frobenius split cubic surfaces in Section \S 5. If $\mathcal X$ is not Frobenius split, then no hyperplane section is Frobenius split, by Theorem \ref{deform}. Invoking 
   our classification of the non-Frobenius split cubic curves (Theorem \ref{3variable} and Lemma \ref{2variable}), we see that no hyperplane section of $\mathcal X$ can be a union of  three lines, unless those lines meet at a point. So $\mathcal X$ contains no triangles, and   is projectively equivalent to the Fermat cubic surface by Corollary \ref{EckSmooth}.{\footnote{  
  Alternatively,  we can instead use Beauville's theorem  to come to the same conclusion: 
  Since every smooth hyperplane section of $\mathcal X$ is a supersingular elliptic curve of characteristic two, all such sections are isomorphic (\cite[p.~ 260]{Husemoller}), so 
  Theorem \ref{beauville} implies $f$ is the Fermat cubic.}}

{}
  Next, assume the rank of $A$ is one. Here the rows of $A$ are multiples of some fixed linear form $L$, say $L_i = \lambda_i L$.  In this case, we can rewrite $f= L (L')^2$ where $L'$ is some other linear form.  There are two cases,
  either $f= x^3$ or $f = x^2y$, up to changing coordinates. Both of these are degenerate.

{}
Now, assume the rank of $A$ is two.  Without loss of generality, we assume  $L_1, L_2$ are linearly independent and write $L_3 = \lambda_1 L_1 + \lambda_2 L_2$ and $L_4 = \mu_1 L_1 + \mu_2 L_2$. Write 
    $$
    f = x_1^2L_1 + x_2^2 L_2 + x_3^2(\lambda_1 L_1 + \lambda_2 L_2) + x_4^2(\mu_1 L_1 + \mu L_2).$$ 
    Reorganizing, we have
    $$
   f =  (x_1+ \lambda_ 1^{1/2}x_3 + \mu_1^{1/2} x_4)^2L_1 + (x_2+  \lambda_2^{1/2} x_3 +  \mu_2^{1/2}x_4)^2 L_2 .
   $$
   which we re-write as 
   $$
   f =   x_1^2L_1 + x_2^2 L_2.
 $$
 for some (new) linear forms $x_1, x_2$. 
  Now, if the linear forms $x_1, x_2, L_1, L_2$  are not linearly independent, then changing coordinates, $f$ can be written as a cubic in three variables, and  the cubic is degenerate. Thus, without loss of generality, we have 
 $f =    x_1^2x_3 + x_2^2x_4.$

  {}
Next, assume the rank of $A$ is three.
 Assume $L_1, L_2, L_3$ are linearly independent, and write $L_4 = \lambda_1 L_1 + \lambda_2 L_2 + \lambda_3 L_3$. Rewrite $f$ as
    $$f = L_1\underset{m_1^2}{\underbrace{(x_1^2 + \lambda_1 x_4^2)}} + L_2 (\underset{m_2^2}{\underbrace{x_2^2 + \lambda_2 x_4^2}}) + L_3 (\underset{m_3^2}{\underbrace{x_3^2 + \lambda_3 x_4^2}}).$$
    Notice that $m_1, m_2, m_3$ are linearly independent, and thus after a change of coordinates we may write $f$ as
    $$f = L_1 x_1^2 + L_2 x_2^2 + L_3 x_3^2.$$
   Again, our non-degeneracy hypothesis implies that  $x_4$ appears in at least one of the $L_i$, so that 
 after a change of coordinates we may assume $L_1 = x_4$. So
    $$f = x_4x_1^2 + (a_{21} x_1 + \cdots + a_{24} x_4) x_2^2 + (a_{31} x_1 + \cdots + a_{34} x_4) x_3^2,$$
    for some constants $a_{ij}$. So
    $$A = \left(\begin{array}{cccc} 0 & 0 & 0 & 1 \\
    a_{21} & a_{22} & a_{23} & a_{24} \\
    a_{31} & a_{32} & a_{33} & a_{34} \\
    0 & 0 & 0 & 0
    \end{array} \right)$$
    Notice that because $A$ has rank $3$, the first three columns must span a space of dimension two.     
    We can now do changes of coordinates to get $a_{24} = a_{34} = 0$ without affecting the basic form of the matrix. For example, using Remark \ref{rowops}, we can 
  add a suitable multiple of the first row to the second row   to make $a_{24} = 0$; the corresponding column operation adds a multiple of column one to column two, changing only the second and third rows of the second column. Similarly, we can make $a_{34} = 0$, so that we can assume the matrix $A$ has the form
    $$A = \begin{pmatrix} 0 & 0 & 0 & 1 \\
    a_{1} & a_{2} & a_{3} & 0 \\
    b_{1} & b_{2} & b_{3} & 0 \\
    0 & 0 & 0 & 0
    \end{pmatrix}.$$
    The submatrix
    $$B = \begin{pmatrix}
    	a_{2} & a_{3} \\  b_{2} & b_{3}
    \end{pmatrix}$$
    is non-zero, since $A$ has rank $3$. This submatrix  represents  a cubic in  the two variables $x_2, x_3$. By  Proposition  \ref{2variable}, a linear change of coordinates involving only $x_2$ and $x_3$ will bring  the submatrix $B$ to one of the standard forms
     $$\begin{pmatrix} 1 & 0 \\ 0 & 0
    \end{pmatrix}, \,\,\,\,\,\
    \begin{pmatrix} 0 & 1 \\ 0 & 0
    \end{pmatrix}, \,\,\,\,\,\,\,
      \begin{pmatrix} 1 & 0  \\ 0 & 1
    \end{pmatrix}
$$
without changing any of the zero entries of the matrix $A$. 
 That is, $A$ can be assumed to be in  one of the following three forms:
$$
 \begin{pmatrix} 0 & 0 & 0 & 1 \\
    a & 1 & 0 & 0 \\
    b & 0 & 0 & 0 \\
    0 & 0 & 0 & 0
    \end{pmatrix}
    \,\,\,\,\,\,\,\,\,\,\,    \,\,\,\,\,\,\,\,\,\,\,
 \begin{pmatrix} 0 & 0 & 0 & 1 \\
    a & 0 & 1 & 0 \\
    b & 0 & 0 & 0 \\
    0 & 0 & 0 & 0
       \end{pmatrix}
    \,\,\,\,\,\,\,\,\,\,\,    \,\,\,\,\,\,\,\,\,\,\,
 \begin{pmatrix} 0 & 0 & 0 & 1 \\
    a & 1 & 0 & 0 \\
    b & 0 & 1 & 0 \\
    0 & 0 & 0 & 0
    \end{pmatrix}.
$$

    {}
 In each of these three cases, we can do row and column operations (according to the rules prescribed by Remark \ref{rowops})  to bring these to the following three forms, respectively:
 
   {}
$$
\begin{pmatrix} 
    0 & 0 & 0 & 1 \\
    0 & 1 & 0 & 0 \\
    1 & 0 & 0 & 0 \\
    0 & 0 & 0 & 0
    \end{pmatrix}
        \,\,\,\,\,\,\,\,\,\,\,    \,\,\,\,\,\,\,\,\,\,\,
\begin{pmatrix} 
    0 & 0 & 0 & 1 \\
    0 & 0 & 1 & 0 \\
    1 & 0 & 0 & 0 \\
    0 & 0 & 0 & 0
    \end{pmatrix}
        \,\,\,\,\,\,\,\,\,\,\,    \,\,\,\,\,\,\,\,\,\,\,
\begin{pmatrix} 
    0 & 0 & 0 & 1 \\
    0 & 1 & 0 & 0 \\
   	0 & 0 & 1 & 0 \\
    0 & 0 & 0 & 0
    \end{pmatrix}.
    $$
    
 {}
For example, to transform the first matrix,  we first  add a multiple of the second column to the first column (along with the corresponding row operation) to  obtain a matrix of the form
    $$\begin{pmatrix} 
    0 & \star & 0 & 1 \\
    0 & 1 & 0 & 0 \\
    b & 0 & 0 & 0 \\
    0 & 0 & 0 & 0
    \end{pmatrix},$$
    where $b\neq 0$ (as the rank is 3). 
    Since the last row is zero, we can simply add a multiple of the last column to the second column to eliminate $\star$, obtaining
    $$\begin{pmatrix} 
    0 & 0 & 0 & 1 \\
    0 & 1 & 0 & 0 \\
    b & 0 & 0 & 0 \\
    0 & 0 & 0 & 0
    \end{pmatrix}.$$
    Now multiplying the third row/column appropriately, we get
    $$\begin{pmatrix} 
    0 & 0 & 0 & 1 \\
    0 & 1 & 0 & 0 \\
    1 & 0 & 0 & 0 \\
    0 & 0 & 0 & 0
    \end{pmatrix}.$$

{} 
       Summarizing, the three matrices above correspond, respectively, to the following three cubic forms:
\begin{align*}
F_1 &= x_1^2x_4 + x_2^3 + x_1 x_3^2 \\ 
F_2 &= x_1^2x_4 + x_2^2x_3+ x_1 x_3^2 \\
F_3 &= x_1^2x_4 + x_2^3 +  x_3^3.
\end{align*}

    We have now completed step (1) of the proof of Theorem \ref{main}. Step (2),  showing that the five normal forms we have identified are all distinct, remains. For this, we need only compare those of the same rank, so it suffices to prove the following:

\begin{lemma}\label{distinct}
The rings corresponding to the three choices of $F_1, F_2, F_3$ above are pair-wise non-isomorphic.
\end{lemma}

\begin{proof}
The singular locus for each of $F_1, F_2, F_3$ is defined by the ideal $\langle x_1, x_2, x_3\rangle$. 
Therefore, any change of variables that sends $F_i$ to $F_j$ for $i, j \in \{1, \, 2, \, 3\}$ has to send $\langle x_1, x_2, x_3\rangle$ to $\langle x_1, x_2, x_3\rangle$.
This means any such change of variables, without loss of generality,  has the form 
$$
g \left[\begin{array}{c} x_1 \\ x_2 \\ x_3 \\ x_4 \\ \end{array}\right] 
=\left[\begin{array}{c} L_1 \\ L_2 \\ L_3 \\ L_4 \\ \end{array}\right]=\left[\begin{array}{c} a_1x_1+a_2x_2+a_3x_3 \\ b_1x_1+b_2x_2+b_3x_3\\ c_1x_1+c_2x_2+c_3x_3 \\ d_1 x_1+d_2x_2+d_3x_3+ x_4 \\ \end{array}\right].
$$
We claim that any such isomorphism $g$ taking $F_i$ to $F_j$  must fix $x_1$. To see this, we substitute $L_1, L_2, L_3, L_4$ in for $x_1, x_2, x_3, x_4$ into  $F_i$ and compare to $F_j$. The terms that contain $x_4$ in $g\cdot F_i$ are $L_1^2 x_4$ whereas in $F_j$ there is only $x_1^2x_4$.  This means that  $L_1 = x_1$, so that any linear change of coordinates taking $F_i$ to $F_j$ must  fix $x_1$.

It is now easy to see that the hypersurfaces defined by  $F_1$, $F_2$ and $F_3$ are projectively distinct. First, 
suppose  that 
$g\cdot F_1=F_2$.  Since $g$ fixes $x_1$, it  induces isomorphisms
$$
\frac{k[x_1, x_2, x_3, x_4]}{\langle x_1, F_1 \rangle} \cong\frac{k[x_1, x_2, x_3, x_4]}{\langle g\cdot x_1, g\cdot F_1 \rangle} \cong \frac{k[x_1, x_2, x_3, x_4]}{\langle  x_1, F_2 \rangle }
$$
or
$$
\frac{k[x_2, x_3, x_4]}{\langle  x_2^3 \rangle }\cong \frac{k[x_2, x_3, x_4]}{ \langle x_2^2x_3 \rangle }.
$$
But these are obviously not isomorphic, as they exhibit different index of nilpotency. So $F_1$ and $F_2$ cut out projectively distinct hypersurfaces in $\mathbb P^3$.

Similarly, we can see that  if there is a linear change of coordinates $g$ such that  $g\cdot F_2=F_3$, we would get an isomorphism
$$
\frac{k[x_2, x_3, x_4]}{\langle x_2^2x_3 \rangle }\cong \frac{k[x_2, x_3, x_4]}{\langle  x_2^3+x_3^3 \rangle },
$$
again a contradiction since the ring on the right is reduced. Finally, 
 if we assume $g\cdot F_3=F_1$, we would get an isomorphism
$$
\frac{k[x_2, x_3, x_4]}{\langle  x_2^3 \rangle }\cong \frac{k[x_2, x_3, x_4]}{\langle  x_2^3+x_3^3 \rangle },
$$
which is a contradiction for the same reason. \end{proof}
\end{proof}

\begin{remark}\label{p>2} We conclude by sketching an alternative proof of the fact that a smooth cubic surface in characteristic $p>2$ is Frobenius split (\cite[5.5]{hara.rational-singularities}), using Beauville's theorem. Fix 
a smooth cubic surface  $\mathcal X$ over an algebraically closed field of char $p>0$.   Choose a pencil of hyperplane sections, blown up at the line they meet. This gives a polarized family of elliptic curves over a dense open set of the affine line (after throwing away any singular member). Since a coarse moduli space exists \cite[p.~ 206]{MumfordSuominen}, the base of this family must map to the $j$-line, and so its image is connected. Now, if $\mathcal X$ is not Frobenius split, the members of this family of elliptic curves are all supersingular. But there are only finitely many supersingular $j$-invariants, so in fact, the map to the $j$-line is a constant map. So all the smooth hyperplane sections of $\mathcal X$ are isomorphic. By Theorem \ref{beauville}, our cubic hypersurface must be projectively equivalent to  one defined by $\sum_{i}x_i^{p^e+1}$, which means 
that  $\mathcal X$ must be projectively equivalent to the Fermat cubic and the characteristic must be two. 
\end{remark}

\section{Frobenius Splitting and Lines on Cubic Surfaces}\label{geometry}

{}
\noindent

In this section, we use our classification to deduce a characterization of  Frobenius splitting of cubic surfaces   in terms of their configuration of lines. Our results are  valid whether or not the surface is smooth.  In the smooth case, our results overlap with those of Hara  \cite[5.5]{hara.rational-singularities},
 but our approach different, and we think, more elementary.

Recall that a point on a cubic surface is an \emph{Eckardt point} if it is an intersection point  of three distinct lines on the surface. Generically, we expect three lines in a plane to form a 
``triangle''---that is, we expect that they do {\it not} meet in a single point. Thus the cubic surfaces that contain Eckardt points are more special than a generic cubic surface. 

{}
Non-Frobenius split cubic surfaces have the most degenerate possible configurations of lines and thus the maximal possible number of Eckardt points:

\begin{theorem}\label{eckardt}
Consider a (possibly singular)    cubic surface $\mathcal X$ over a field of characteristic two that is not Frobenius split. 
If two lines $\ell_1$ and $\ell_2$ on $\mathcal X$  meet at a point, then either one of the lines is a double line,
or the intersection is an Eckardt point. That is, a non-Frobenius split cubic surface of characteristic two contains no triangles.
\end{theorem}

\begin{proof} 
Fix two  intersecting lines $\ell_1$ and $\ell_2$ on a  (possibly singular, or even degenerate) cubic surface $\mathcal X$ that is not Frobenius split.
Let $H$ be the plane spanned by $\ell_1$ and $\ell_2$. Since $\mathcal X$ has degree three, we know $\mathcal X \cap H$ contains exactly one more line, and   the coordinate ring of the hyperplane section   $\mathcal X \cap H$ is isomorphic to $k[x, y, z]/\langle xyL\rangle$ where  $L$ is a linear form.  By Theorem \ref{deform}, the ring $ k[x, y, z]/\langle xyL\rangle$ cannot be $F$-pure, for if it were,  the  coordinate ring of $\mathcal X$ would also be $F$-pure.
By our classification theorem for non-Frobenius split cubics in three variables, the cubic curve $\mathcal X \cap H$ must therefore be degenerate, since
our list in Theorem \ref{3variable} contains no triple of lines. That is,
the third linear form $L$ must lie in the span of $\{x, y\}$. So either the third line is one of $\ell_1$ or $\ell_2$ (this is the case where $L$ is a scalar multiple of $x$ or $y$) or the three lines are distinct, and 
meet at the Eckardt point $[0:0:1]$. 
\end{proof}

We deduce a cute corollary in the smooth case:

\begin{corollary}\label{EckSmooth} A smooth cubic surface of characteristic two  is Frobenius split {\it unless} it is ``triangle-free''---that is, unless {\it  each and every} collection of coplanar lines on it meet at one point. 
In particular, the only smooth cubic surface that is not Frobenius split (up to projective change of coordinates) is the Fermat cubic defined by $x^3+y^3+z^3+w^3$ in characteristic two. 
\end{corollary}

\begin{remark}
Corollary \ref{EckSmooth} holds as stated in any positive characteristic, but somewhat vacuously: in characteristic $p>2$,  every cubic surface   is Frobenius split (\cite[5.5]{hara.rational-singularities}; see also Remark \ref{p>2}) and there are no triangle-free cubic surfaces. The latter statement is likely well-known but follows easily from our classical argument below.
\end{remark}

The second statement of
Corollary  \ref{EckSmooth} follows immediately from  Theorem \ref{eckardt} using the main theorem of \cite{Homma} geometrically  characterizing the Fermat cubic surface in characteristic two. However,  we include two straightforward proofs, the first quite classical (and reproving part of \cite[1.1]{Homma}) and the second using Beauville's theorem (Theorem \ref{beauville}).

\begin{proof}[Classical Proof of Corollary \ref{EckSmooth}] The first statement is immediate from the theorem, since each tri-tangent planar section of a smooth cubic surface  is either a triangle or a triple of lines meeting at an Eckardt point. Thus for the second statement,  because  the Fermat cubic surface is not Frobenius split (Theorem \ref{fedder}), it suffices to show that
 {\it  there is at most one cubic  surface without  any triangles.}{\footnote{Our argument will also show that such a surface can exist only in characteristic two.}}
  
Recall that a cubic surface can be described  as $\mathbb P^2$ blown up at six points $\{p_1, \dots, p_6\}$, with no three co-linear, and no five on a conic (for example, see \cite[V \S4]{hartshorne-algebraic-geometry}). 
Each pair of these points determines a line $\ell_{ij}$ on the cubic surface $\mathcal X$---namely the birational transform  on $\mathcal X$ of the line  through $p_i$ and $p_j$ in $\mathbb P^2$. Note that the lines $\ell_{ij}$ and $\ell_{kl}$ intersect on $\mathcal X$ if and only if $\{i, j\} \cap \{k, l\} = \emptyset.$ 
We claim that the fact that these fifteen lines can never form a triangle on $\mathcal X$  imposes so many conditions that the configuration of six points is uniquely determined.

To see this, choose coordinates so that $p_1=[1:0:0], \, p_2=[0:1:0], \, p_3=[0:0:1]$, and $p_4=[1:1:1]$. 
These determine  three pairs of intersecting lines on the cubic, whose intersection points are $[1:1:0], \, [0:1:1]$, and $[1:0:1]$.
The line $\ell_{56}$ must intersect both lines in each of these intersecting pairs, and hence $\ell_{56}$ must pass through each of these three points, for otherwise we would have a triangle on the cubic surface.  Thus $\ell_{56}$ is uniquely determined{\footnote{This can happen {\it only} in characteristic two! The three points are not  colinear in any other characteristic.}}
 as  the line $\{x+y+z=0\}$. So  $p_5=[a:a+1:1]$ and $p_6=[b:b+1:1]$, for some non-zero scalars $a, b$. 
Because $\ell_{25},  \ell_{46}$ and $ \ell_{13} $($= \mathbb V(y)$) are concurrent,
 we compute that 
 $$
\ell_{25}\cap \ell_{13} =   \ell_{46} \cap \ell_{13}  \,\,\,\,\,\,\, {\text{ so that }} \,\,\,\,     [a: 0:  1] =[1: 0 : b].
$$
 Likewise, 
  because the lines $\ell_{15}, \ell_{46}$ and $ \ell_{23} $($= \mathbb V(x)$) are concurrent, 
 we have
 $$
 \ell_{15}\cap \ell_{23} =   \ell_{46} \cap \ell_{23}  \,\,\,\,\,\,\, {\text{ so that }} \,\,\,\,     [0: a+1:  1] =[0:1:b+1].
 $$
 Therefore, the constants $a$ and $b$ satisfy the relations  $ab=1$ and $a+b=1$. 
 This determines  $p_5$ and $p_6$ as $[\omega:  \omega+1: 1] $ and $[\omega^2, \omega^2+1: 1]$, where $\omega$ is a primitive cube root of unity.
\end{proof}

\begin{proof}[Alternate Proof of Corollary]\label{alt}
Alternatively,  the second statement in Corollary \ref{EckSmooth} can be deduced from  Beauville's theorem.  We know that if $\mathcal X$ is not Frobenius split, then the same is true for every hyperplane section by Theorem \ref{deform}. So the smooth hyperplane sections of $\mathcal X$ are all supersingular elliptic curves of characteristic two
(by  Example \ref{fedder-elliptic: R}), and hence isomorphic  \cite[p.~~ 260]{Husemoller}.  Now Beauville's theorem (Theorem \ref{beauville})  implies that the cubic surface $\mathcal X$ is projectively equivalent to $x^3+y^3+z^3+w^3$.
\end{proof}

{\small
\bibliographystyle{amsalpha}
\bibliography{bibdatabase}

\providecommand{\bysame}{\leavevmode\hbox to3em{\hrulefill}\thinspace}
\providecommand{\MR}{\relax\ifhmode\unskip\space\fi MR }
\providecommand{\MRhref}[2]{%
  \href{http://www.ams.org/mathscinet-getitem?mr=#1}{#2}
}
\providecommand{\href}[2]{#2}
\begin{thebibliography}{Hom97}

\bibitem[Bea90]{beauville.sur_les_hypersurfaces}
Arnaud Beauville, \emph{Sur les hypersurfaces dont les sections hyperplanes
  sont \`a module constant}, The {G}rothendieck {F}estschrift, {V}ol. {I},
  Progr. Math., vol.~86, Birkh\"{a}user Boston, Boston, MA, 1990, With an
  appendix by David Eisenbud and Craig Huneke, pp.~121--133. \MR{1086884}

\bibitem[BK05]{BK}
Michel Brion and Shrawan Kumar, \emph{Frobenius splitting methods in geometry
  and representation theory}, Progress in Mathematics, vol. 231, Birkh\"{a}user
  Boston, Inc., Boston, MA, 2005. \MR{2107324}

\bibitem[DD19]{dolgachev-duncan.automorphisms}
I.~Dolgachev and A.~Duncan, \emph{Automorphisms of cubic surfaces in positive
  characteristic}, Izv. Ross. Akad. Nauk Ser. Mat. \textbf{83} (2019), no.~3,
  15--92. \MR{3954305}

\bibitem[Fed83]{fedder.F-purity}
Richard Fedder, \emph{{$F$}-purity and rational singularity}, Trans. Amer.
  Math. Soc. \textbf{278} (1983), no.~2, 461--480. \MR{701505}

\bibitem[Har77]{hartshorne-algebraic-geometry}
Robin Hartshorne, \emph{Algebraic geometry}, Springer-Verlag, New
  York-Heidelberg, 1977, Graduate Texts in Mathematics, No. 52. \MR{0463157}

\bibitem[Har98]{hara.rational-singularities}
Nobuo Hara, \emph{A characterization of rational singularities in terms of
  injectivity of {F}robenius maps}, Amer. J. Math. \textbf{120} (1998), no.~5,
  981--996. \MR{1646049}

\bibitem[Hir85]{Hirschfeld}
J.~W.~P. Hirschfeld, \emph{Finite projective spaces of three dimensions},
  Oxford Mathematical Monographs, The Clarendon Press, Oxford University Press,
  New York, 1985, Oxford Science Publications. \MR{840877}

\bibitem[Hom97]{Homma}
Masaaki Homma, \emph{A combinatorial characterization of the {F}ermat cubic
  surface in characteristic {$2$}}, Geom. Dedicata \textbf{64} (1997), no.~3,
  311--318. \MR{1440564}

\bibitem[HR74]{hochster+roberts-rings-of-invariants}
Melvin Hochster and Joel~L. Roberts, \emph{Rings of invariants of reductive
  groups acting on regular rings are {C}ohen-{M}acaulay}, Advances in Math.
  \textbf{13} (1974), 115--175. \MR{347810}

\bibitem[HR76]{HR76}
Melvin Hochster and Joel~L Roberts, \emph{The purity of the {F}robenius and
  local cohomology}, Advances in Mathematics \textbf{21} (1976), no.~2, 117 --
  172.

\bibitem[Hus04]{Husemoller}
Dale Husem\"{o}ller, \emph{Elliptic curves}, second ed., Graduate Texts in
  Mathematics, vol. 111, Springer-Verlag, New York, 2004, With appendices by
  Otto Forster, Ruth Lawrence and Stefan Theisen. \MR{2024529}

\bibitem[Lan56]{lang.algebraic_groups}
Serge Lang, \emph{Algebraic groups over finite fields}, Amer. J. Math.
  \textbf{78} (1956), 555--563. \MR{86367}

\bibitem[MR85]{mehta-ramanathan-frobenius-splitting}
V.~B. Mehta and A.~Ramanathan, \emph{Frobenius splitting and cohomology
  vanishing for {S}chubert varieties}, Ann. of Math. (2) \textbf{122} (1985),
  no.~1, 27--40. \MR{799251}

\bibitem[MS72]{MumfordSuominen}
David Mumford and Kalevi Suominen, \emph{Introduction to the theory of moduli},
  Algebraic geometry, {O}slo 1970 ({P}roc. {F}ifth {N}ordic {S}ummer-{S}chool
  in {M}ath.), 1972, pp.~171--222. \MR{0437531}

\bibitem[Smi97]{smith.vanishing}
Karen~E. Smith, \emph{Vanishing, singularities and effective bounds via prime
  characteristic local algebra}, Algebraic geometry---{S}anta {C}ruz 1995,
  Proc. Sympos. Pure Math., vol.~62, Amer. Math. Soc., Providence, RI, 1997,
  pp.~289--325. \MR{1492526}

\bibitem[Smi00]{smithGlobalF-regularityGITQuotientsFanoVarieties}
\bysame, \emph{Globally {F}-regular varieties: applications to vanishing
  theorems for quotients of {F}ano varieties}, Michigan Math. J. \textbf{48}
  (2000), 553--572, Dedicated to William Fulton on the occasion of his 60th
  birthday. \MR{1786505}

\end{thebibliography}
}

\end{document}